\documentclass[reqno]{article}
\usepackage{amsmath,amsfonts,amssymb,amsthm}
\usepackage{graphics,color,epsfig, enumerate}

\newtheorem{theorem}{Theorem}[section]
\newtheorem{lemma}[theorem]{Lemma}

\newtheorem{corollary}[theorem]{Corollary}
\newtheorem{observation}[theorem]{Observation}
\newtheorem{conjecture}{Conjecture}

\newtheorem*{remark}{Remark}

\begin{document}

\title{A stability result for the union-closed size problem}

\author{Tom Eccles
 \thanks{Department of Pure Mathematics and Mathematical Statistics,
Wilberforce Road, Cambridge CB3 0WB, UK; +44 (0)1223 765921; te227@cam.ac.uk}
\thanks{This research was supported by EPSRC}
}
\date{\today}
\maketitle
\begin{abstract}
A family of sets is called union-closed if whenever $A$ and $B$ are sets of the family, so is $A\cup B$. The long-standing union-closed conjecture states that if a family of subsets of $[n]$ is union-closed, some element appears in at least half the sets of the family. A natural weakening is that the union-closed conjecture holds for large families; that is, families consisting of at least $p_02^n$ sets for some constant $p_0$. The first result in this direction appears in a recent paper of Balla, Bollob\'as and Eccles \cite{BaBoEc}, who showed that union-closed families of at least $\frac{2}{3}2^n$ sets satisfy the conjecture --- they proved this by determining the minimum possible average size of a set in a union-closed family of given size. However, the methods used in that paper cannot prove a better constant than $\frac{2}{3}$. Here, we provide a stability result for the main theorem of \cite{BaBoEc}, and as a consequence we prove the union-closed conjecture for families of at least $(\frac{2}{3}-c)2^n$ sets, for a positive constant $c$.
\end{abstract}
\section{Introduction}
We shall be concerned with finite families of finite sets; as often, we shall assume that such a family is a subset of $\mathcal P(n)=\mathcal P([n])$ for some $n$, where $\mathcal P$ denotes the powerset and $[n]=\{1,\dots,n\}$. For $\mathcal A \subseteq \mathcal P(n)$, we call $\mathcal A$ \emph{union-closed} if for any two elements $A$ and $B$ of $\mathcal A$ the set $A\cup B$ is also in $\mathcal A$. For $i\in \mathbb N$, the \emph{degree of $i$ in $\mathcal A$}, denoted $\mathrm{deg}_{\mathcal A}(i)$, is simply
\[
|\{A\in \mathcal A: i\in A\}|.
\]
The \emph{union-closed conjecture}, often attributed to Frankl \cite{Duf}, states that if $\mathcal A$ is a union-closed family other than $\{\emptyset\}$ then there is some $i$ with $\mathrm{deg}_{\mathcal A}(i)\ge |\mathcal A|/2$.

A related problem is the \emph{union-closed size problem}, which asks how small the sets of a union-closed family can be. For a finite family $\mathcal A\subseteq \mathcal P(n)$, we define the \emph{total size} of $\mathcal A$ to be
\[
 ||\mathcal A|| = \sum_{A\in \mathcal A}|A|.
\]
Then the union-closed size problem asks what is the value of
\[
 f(m) = \min ||\mathcal A||,
\]
where the minimum runs over union-closed families which consist of $m$ sets. This problem was first addressed by Reimer \cite{Rei} in 2003, who proved that
\[
 f(m) \ge \frac{m}{2} \log_2 m.
\]
Recently, Balla, Bollob\'as and Eccles \cite{BaBoEc} settled the union-closed size problem entirely, determining the exact value of $f(m)$ for all $m$. We denote by $\mathcal I(m)$ the initial segment of the colex order on $\mathbb N^{(<\infty)}$ of length $m$; this order shall be defined fully in Section \ref{sec_defs}.
\begin{theorem}\label{thm_colex_bound}
Let $m$ be a positive integer, and let $n$ be the unique integer with $2^{n-1} < m \le 2^n$. Set $m' = 2^n-m$. Then
\[
f(m) = ||\mathcal P(n)|| - ||\mathcal I(m')|| - m'.
\]
In particular, if $\mathcal A$ is a counterexample to the union-closed conjecture in $\mathcal P(n)$ with $|\mathcal A|=m$ then $f(m)<nm/2$, and so
\[
||\mathcal I(m')||+m' > n2^n/2 - nm/2 = nm'/2.
\]
\end{theorem}
\noindent
The extremal family $\mathcal A$ for the first part of the theorem has $\mathcal P(n)\setminus \mathcal A = \{B\cup \{n\}: B\in \mathcal I(m')\}$. Through bounding $||\mathcal I(m')||$, this result is sufficient to prove the union-closed conjecture if $|\mathcal A|$ is large --- in fact the following bound is given in \cite{BaBoEc}.
\begin{corollary}\label{cor_old_bound}
The union-closed conjecture holds for all union-closed families $\mathcal A\subseteq \mathcal P(n)$ with $|\mathcal A| \ge \frac{2}{3}2^n$.
\end{corollary}
However, this is as far as one can go considering only averaging arguments --- if $m < \frac{2}{3}2^n$, then $f(m) < mn/2$. From this, one might reasonably assume that the constant $\frac{2}{3}$ in Corollary \ref{cor_old_bound} is hard to improve to any constant $\frac{2}{3}-\epsilon$. But in the extremal examples for $f(m)$, the family $\mathcal A$ is very asymmetric --- indeed, there is a single element which is in every set of $\mathcal P(n)\setminus \mathcal A$ --- and so $\mathcal A$ is in a sense far away from being a counterexample to the union-closed conjecture. In this paper, we prove a stability result for the union-closed size problem for union-closed families $\mathcal A\subseteq \mathcal P(n)$ with $|\mathcal A|\ge 2^{n-1}$. Roughly speaking, we show that if $||\mathcal A||$ is close to the maximum possible then $\mathcal P(n)\setminus \mathcal A$ has an element of high degree --- this result is Theorem \ref{thm_stability}. This enables us to extend Theorem \ref{thm_colex_bound}.
\begin{theorem}\label{thm_main}
There is a positive constant $c_1$ such that if $\mathcal A$ is a counterexample to the union-closed conjecture in $\mathcal P(n)$, and $\mathcal B= \mathcal P(n)\setminus \mathcal A$ with $|\mathcal B| = m$, then
\[
 ||\mathcal I(m)|| > m(n/2-1+c_1).
\]	
\end{theorem}
\noindent
Using simple bounds on $||\mathcal I(m)||$, this extends slightly the range where we can prove the union-closed conjecture.
\begin{corollary}\label{cor_main}
There is a positive constant $c_2$ such that the union-closed conjecture holds for all union-closed familes $\mathcal A\subseteq \mathcal P(n)$ with $|\mathcal A| \ge 2^n(2/3 - c_2)$.
\end{corollary}
In fact, we shall prove these theorems with bounds of $c_1\ge 1/24$ and $c_2 \ge 1/104$.

The rest of the paper is organised as follows. In Section \ref{sec_defs}, we define the concepts needed in the proofs of our main theorems --- in particular down-compressions and simply rooted families, which shall be at the heart of our argument. In Section \ref{sec_state} we state Theorem \ref{thm_stability}, our stability result for Theorem \ref{thm_colex_bound}. In Section \ref{sec_pf_thm}, we prove Theorem \ref{thm_stability}, and use it to prove Theorem \ref{thm_main}. In Section \ref{sec_pf_cor} we bound $||\mathcal I(m)||$, proving Corollary \ref{cor_main} from Theorem \ref{thm_main}. In Section \ref{sec_refine} we prove a slightly stronger form of Theorem \ref{thm_stability}, which improves the constants $c_1$ and $c_2$ a little --- this is left out of the main proof for the sake of clarity.

\section{Definitions}\label{sec_defs}
In this section, we recall some concepts used by Reimer \cite{Rei} and Balla, Bollob\'as and Eccles \cite{BaBoEc} in their work on the union-closed size problem. Central to both of those papers are \emph{compressions}. Up- and down-compressions are by now standard; see for example Bollob\'as and Leader \cite{BoLe}. For a family $\mathcal B\subseteq \mathcal P(n)$ and $i\in [n]$, we define the \emph{down-compression of $\mathcal B$ in direction $i$}, denoted $d_i(\mathcal B)$, by defining
\[
d_{(i,\mathcal B)}(B)=
\begin{cases}
B-i: i\in B,\, B-i\notin \mathcal B\\
B: \mathrm{ otherwise,}
\end{cases}
\]
and $d_i(\mathcal B) = \{d_{(i,\mathcal B)}(B): B\in \mathcal B\}$. A down-compression of a family $\mathcal B$ is equivalent to an up-compression on its complement in $\mathcal P(n)$, in that
\begin{equation}\label{eq_comp_equiv}
 \mathcal P(n)\setminus d_i(\mathcal B) = u_i(\mathcal P(n)\setminus \mathcal B),
\end{equation}
where $u_i$ is the up-compression in direction $i$, defined analogously to $d_i$. Also, for $\mathcal B\subseteq \mathcal P(n)$ we define $d(\mathcal B)$ to be $d_n\dots d_1(\mathcal B)$, the compression obtained by applying the compressions $d_i$ to $\mathcal B$ for $1\le i \le n$, starting with $d_1$.

 For $B\in \mathcal B$ we define $d_{\mathcal B}(B)$ to be the image of $B$ under the down-compression $d_\mathcal B$; that is, letting $\mathcal B_i$ = $d_i\dots d_1 (\mathcal B)$, we define
\[
 d_\mathcal B(B) = d_{(n, \mathcal B_{n-1})}\dots d_{(2, \mathcal B_1)}d_{(1,\mathcal B)}(B);
\]
so $d_{\mathcal B}(B)$ is the set we get by following $B$ through the compressions $d_i$. Similarly, we shall often want to consider the family $\mathcal B$ after some of the compressions $d_i$ have been applied; to this end we define $D_k(\mathcal B) = d_k\dots d_1(\mathcal B)$, the family after compressing in directions $i$ for $1\le i \le k$, and for $B\in \mathcal B$ we define $D_{(\mathcal B,k)}(B) = d_{(k, \mathcal B_{k-1})}\dots d_{(1,\mathcal B)}(B)$, the image of the set $B$ in $D_k(\mathcal B)$.

Following the approach of \cite{BaBoEc}, we shall view the complement of a union-closed family as a simply rooted family --- this perspective is crucial for our proof of Theorem \ref{thm_main}. We call a family $\mathcal B\subseteq \mathcal P(n)$ \emph{simply rooted} if for every $\emptyset \neq B\in \mathcal B$, there is some $b\in B$ with $[\{b\},B]\subseteq \mathcal B$. The following simple observation was made in \cite{BaBoEc}.

\begin{observation}\label{obs_rooted}
Let $\mathcal A\subseteq \mathcal P(n)$, and $\mathcal B = \mathcal P(n)\setminus \mathcal A$. Then $\mathcal B$ is a simply rooted family if and only if $\mathcal A$ is a union-closed family.
\end{observation}
\begin{proof}
The family $\mathcal A$ is union-closed exactly when for every $B\notin \mathcal A$ we have
\[
\bigcup_{B'\subseteq B,\,B'\in \mathcal A}B'\neq B,
\]
which is in turn true exactly when $[\{b\},B]\subseteq \mathcal B$ for some $b\in B$.
\end{proof}

Finally, we recall the colex order on $\mathbb N^{(<\infty)}$, the collection of finite sets of positive integers, and some of its standard properties. Given $A$ and $B$ sets in $\mathbb N^{(<\infty)}$, we define the colex order $<$ by
\[
 A<B \iff \max(A\triangle B) \in B.
\]
This is a linear order on $\mathbb N^{(<\infty)}$. We write $\mathcal I(m)$ for the initial segment of this order of length $m$; so, for example, $\mathcal I(9) = \{\emptyset, 1, 2, 12, 3, 13, 23, 123, 4\}$, where we write $13$ for the set $\{1,3\}$. Also, a family of sets $\mathcal D$ is called a \emph{down-set} if for every $A\in \mathcal D$ we have $\mathcal P(A)\subseteq \mathcal D$. The following result is a well-known consequence of the fundamental theorem of Kruskal \cite{Kru} and Katona \cite{Kat}.
\begin{lemma}\label{lem_KK}
If $\mathcal D$ is a down-set, then
$||\mathcal D||\le ||\mathcal I(|\mathcal D|)||$. \hfill\qed
\end{lemma}
The other fact which we shall need about initial segments of colex is the following lemma, which is a simple corollary of Lemma \ref{lem_KK} -- see for example \cite{BaBoEc}.
\begin{lemma}\label{lem_colex_sums}
 Let $m_1$ and $m_2$ be positive integers. Then
\[
 ||\mathcal I(m_1)||+||\mathcal I(m_2)|| \le ||\mathcal I(m_1+m_2)|| - \min(m_1,m_2).
\]
\end{lemma}
This can be proved for $m_1\ge m_2$ by applying Lemma \ref{lem_KK} to the down-set $\mathcal I(m_1)\cup\{A+N:A\in \mathcal I(m_2)\}$, for a sufficiently large integer $N$. \hfill\qed

\section{Stability for sizes of simply rooted families}\label{sec_state}
For a simply rooted family $\mathcal B\subseteq \mathcal P(n)$, a set $B\in \mathcal B$ and an element $i\in [n]$, we say that $B$ is \emph{$\mathcal B$-rooted at $i$} if $i\in B$ and the cube $[\{i\},B]$ is contained in $\mathcal B$. Then for a set $S\subseteq [n]$, we define $\mathcal B_S$ to be those sets of $\mathcal B$ which are $\mathcal B$-rooted at some $i\in S$. 

Let $\mathcal B\subseteq \mathcal P(n)$ be a simply rooted family. By Observation \ref{obs_rooted} the family $\mathcal P(n)\setminus \mathcal B$ is union-closed, and so Theorem \ref{thm_colex_bound} gives us
\begin{equation}\label{eq_rooted_bound}
||\mathcal B|| = ||\mathcal P(n)|| - ||\mathcal A|| \le ||\mathcal P(n)||-f(|\mathcal A|) = ||\mathcal I(|\mathcal B|)||+|\mathcal B|.
\end{equation}
For any $m$, the family $\mathcal B=\{B+n:B\in \mathcal I(m)\}$ makes this inequality tight for $n=\lceil log_2(m)\rceil +1$ --- every set in this family is $\mathcal B$-rooted at $n$. In fact, up to isomorphism this is the only simply rooted family of $m$ sets for which equality holds; this is a consequence of the uniqueness of extremal families for $f(m)$, which was proved in \cite{BaBoEc}. In particular, if $||\mathcal B|| = ||\mathcal I(|\mathcal B|)||+|\mathcal B|$ then $\mathcal B_{\{i\}} = \mathcal B$ for some $i\in [n]$. The following result extends this, showing that if $||\mathcal B||$ is close to $||\mathcal I(m)||+m$ then $|\mathcal B_{\{i\}}|$ is large for some $i$.
\begin{theorem}\label{thm_stability}
Let $\mathcal B$ be a simply rooted family in $\mathcal P(n)$ with $|\mathcal B| = m$, and $p\in [0,1]$. Suppose that $|\mathcal B_{\{i\}}| \le p m$ for all $i\in [n]$. Then
\begin{equation*}
||\mathcal B|| \le ||\mathcal I(m)|| + m - m^2(1/12 - p^2/12)/2^n.
\end{equation*}
\end{theorem}
Theorem \ref{thm_stability} provides a stability result for Theorem \ref{thm_colex_bound} for union-closed families $\mathcal A \subseteq \mathcal P(n)$ with $|\mathcal A|\ge 2^{n-1}$. Indeed, let $\mathcal A$ be such a family and set $\mathcal B = \mathcal P(n)\setminus \mathcal A$ with $|\mathcal B|=m$. Since $\mathcal B$ is a simply rooted family by Observation \ref{obs_rooted}, if $|\mathcal B_{\{i\}}| \le p m$ for all $i$ we have
\begin{align*}
||\mathcal A|| &= ||\mathcal P(n)||-||\mathcal B||\\
& \ge ||\mathcal P(n)||-||\mathcal I(m)|| - m + m^2(1/12 - p^2/12)/2^n\\
&= f(|\mathcal A|) + m^2(1/12 - p^2/12)/2^n.
\end{align*}
Hence if $||\mathcal A||$ is close to $f(|\mathcal A|)$, some element of $[n]$ appears in nearly all the sets of $\mathcal B$.

\section{Proofs of main theorems}\label{sec_pf_thm}
Now we turn to the proofs of Theorems \ref{thm_stability} and \ref{thm_main}. First we give two definitions we shall need in the proof of Theorem \ref{thm_stability}. For a finite set $B$, let $\delta B = \{B-i : i \in B\}$ be the \emph{shadow of $B$}. Given a simply rooted family $\mathcal B\subseteq \mathcal P(n)$, we call a set $B\in \mathcal B$ a \emph{bad set of $\mathcal B$} if either $\delta B\subseteq \mathcal B$ or $d_{\mathcal B}(B)=B$. We call a set $B\in \mathcal B$ that is not bad a \emph{good set of $\mathcal B$}. 

We now sketch the proof of Theorem \ref{thm_stability}. In Lemmas \ref{lem_no_falls} and \ref{lem_full_sh} we shall show that if $\mathcal B$ has many bad sets then $||\mathcal B||$ is much less than $||\mathcal I(m)||+m$; as a result, it is enough to show that a simply rooted family satisfying the condition of Theorem \ref{thm_stability} has many bad sets. Given such a simply rooted family $\mathcal B$, we shall then write $\mathcal B$ as $\mathcal B_S\cup \mathcal B_T$, where $S\cup T$ is a partition of $[n]$. Since no $\mathcal B_{\{i\}}$ is too large, we can do this so both $\mathcal B_S$ and $\mathcal B_T$ are fairly large. If their intersection $|\mathcal B_S\cap \mathcal B_T|$ is large, then we conclude that $\mathcal B$ has many bad sets, since all the sets of $\mathcal B_S\cap \mathcal B_T$ are $\mathcal B$-rooted at two elements of $[n]$, and so are bad sets of $\mathcal B$. If, on the other hand, $|\mathcal B_S\cap \mathcal B_T|$ is small, we shall show in Corollary \ref{cor_lower_b} that $\mathcal B$ still has many bad sets. We prove this by considering the down-sets $d(\mathcal B_S)$ and $d(\mathcal B_T)$; since these are large down-sets in $\mathcal P(n)$, they have a large intersection, and in Lemma \ref{lem_split_rooted} we shall show that sets in this intersection correspond to sets in either $\mathcal B_S \cap \mathcal B_T$ or bad sets of $\mathcal B$.

\subsection{Applying down-compressions to simply rooted families}
In Lemma \ref{lem_no_falls}, we shall show that if $\mathcal B$ has many sets with $d_\mathcal B(B)=B$ then $||\mathcal B||$ is small. In order to prove this lemma, we first recall some results of Reimer \cite{Rei} on union-closed families, restating them in terms of simply rooted families.

\begin{lemma}\label{lem_Rei_basics}
 Let $\mathcal B\subseteq \mathcal P(n)$ be a simply rooted family. Then
\begin{enumerate}
\item \label{it_end_down} $d(\mathcal B)$ is a down-set,
\item \label{it_rooted}for $1\le k\le n$, $D_k(\mathcal B)$ is a simply rooted family.\hfill\qed
\end{enumerate}
\end{lemma}

We now prove some further basic properties of down-compressions on simply rooted families.

\begin{lemma}\label{lem_rooted_basics}
Let $\mathcal B\subseteq \mathcal P(n)$ be a simply rooted family. Then
\begin{enumerate}
\item \label{it_down-set}for $B\in \mathcal B$ and $1\le k\le n$, if $D_{(\mathcal B, k)}(B)\neq B$ then $\mathcal P(D_{(\mathcal B, k)}(B))\subseteq D_k(\mathcal B)$,
\item \label{it_one_fall}for $B\in \mathcal B$, $|B\setminus d_{\mathcal B}(B)|\le 1$.
\end{enumerate}
\end{lemma}
\begin{proof}
Suppose that $D_{(\mathcal B, j)}(B)\neq B$ for some $j\in [n]$; otherwise both parts of the lemma hold for the set $B$. Let $\ell$ be minimal with $D_{(\mathcal B, \ell)}(B) \neq B$. Then $D_{(\mathcal B, \ell)}(B) = B-\ell$, and $B-\ell\notin D_{\ell-1}(\mathcal B)$. Also, by Part \ref{it_rooted} of Lemma \ref{lem_Rei_basics}, $D_{\ell-1}(\mathcal B)$ is simply rooted, and so there is some $i\in B$ such that $[\{i\},B] \subseteq D_{\ell-1}(\mathcal B)$ --- and since $B-\ell \notin D_{\ell-1}(\mathcal B)$, we must have $i = \ell$. Hence $\mathcal P(B-\ell) \subseteq D_\ell(\mathcal B)$, since if a family $\mathcal F$ contains $S+\ell$ for some set $S$ then $d_\ell(\mathcal F)$ contains $S$.

Now, $\mathcal P(B-\ell)$ is a down-set which is contained in $D_\ell(\mathcal B)$, and so any down-compression of the family $D_\ell(\mathcal B)$ fixes every set in $\mathcal P(B-\ell)$. For Part \ref{it_down-set} of the lemma, if $D_{(\mathcal B, k)}(B)\neq B$ for some $k\in [n]$, $k\ge \ell$, and so $D_k(\mathcal B) = d_k\dots d_{l+1}(D_l(\mathcal B))$. Hence we have $D_{(\mathcal B,k)}(B) = B-\ell$ and
\[
\mathcal P(D_{(\mathcal B, k)}(B)) = \mathcal P(B-\ell)\subseteq D_k(\mathcal B),
\]
so Part \ref{it_down-set} holds. For Part \ref{it_one_fall}, we have $d_{\mathcal B}(B)=D_{(\mathcal B,\ell)}(B)=B-\ell$, so $|B\setminus d_{\mathcal B}(B)|=1$.
\end{proof}
From Lemmas \ref{lem_Rei_basics} and \ref{lem_rooted_basics}, we immediately get a bound on the total size of a simply rooted family.
\begin{lemma}\label{lem_no_falls}
Let $\mathcal B$ be a simply rooted family, let $|\mathcal B| = m$, and let $m'$ be the number of sets $B\in \mathcal B$ with $d_{\mathcal B}(B)= B$. Then
\[
||\mathcal B||\le ||\mathcal I(m)|| + m - m'.
\]
\end{lemma}
\begin{proof}
Note that
\[
 ||\mathcal B|| = ||d(\mathcal B)|| + \sum_{B\in \mathcal B} |B\setminus d_{\mathcal B}(B)|.
\]
By Part \ref{it_end_down} of Lemma \ref{lem_Rei_basics}, $d(\mathcal B)$ is a down-set, and so by Lemma \ref{lem_KK} $||d(\mathcal B)||$ is at most $||\mathcal I(m)||$. Also, by Part \ref{it_one_fall} of Lemma \ref{lem_rooted_basics}, $\sum_{B\in \mathcal B} |B\setminus d_{\mathcal B}(B)|$ is exactly $m-m'$, and so the result follows.
\end{proof}

Note that if $\mathcal B$ is a simply rooted family and $B \in \mathcal B$, $|\delta B\setminus \mathcal B|\le 1$. We shall now show that if there are many $B\in \mathcal B$ with the entire shadow of $B$ contained in $\mathcal B$ then Theorem \ref{thm_stability} holds.

\begin{lemma}\label{lem_full_sh}
 Let $\mathcal B$ be a simply rooted family of size $m$, and set $m'$ to be the number of sets $B\in \mathcal B$ with $\delta B \subseteq \mathcal B$. Then
\[
 ||\mathcal B|| \le ||I(m)|| + m-m'.
\]
\end{lemma}
We shall deduce this from a more general result. For $\mathcal B$ a finite family of finite sets, we define the \emph{deficiency of $\mathcal B$}, denoted $\mathrm{def}(\mathcal B)$, to be the number of sets in the shadows of sets $B\in \mathcal B$ that are missing from $\mathcal B$ --- that is,
\[
 \mathrm{def}(\mathcal B)  = \sum_{B\in \mathcal B}|\delta(B)\setminus \mathcal B|.
\]
Then we have the following lemma concerning the total size of a family of given size and deficiency.
\begin{lemma}\label{lem_def}
 Suppose that $\mathcal B$ is a finite family of finite sets in $\mathcal P(n)$, with $|\mathcal B|=m$. Then
\begin{equation*}
 ||\mathcal B|| \le ||\mathcal I(m)|| + \mathrm{def}(\mathcal B).
\end{equation*}
\end{lemma}
We note that Lemma \ref{lem_full_sh} is immediate from this lemma, since if $\mathcal B$ is a simply rooted family then for each $B$ we have $|\delta B \setminus \mathcal B|\le 1$, and there are $m'$ sets $B\in \mathcal B$ with $|\delta B \setminus \mathcal B| = 0$, so $\mathrm{def}(\mathcal B)=m-m'$.
\begin{proof}
We apply induction on $n$. For $n=1$ the result is easily checked. If $n>1$, we define families of sets $\mathcal B_n^+$ and $\mathcal B_n^-$ by
\begin{align*}
\mathcal B^+_n &= \{B\in \mathcal P(n-1):B + n\in \mathcal B\},\,\textrm{and}\\
\mathcal B^-_n &= \{B\in \mathcal P(n-1): B\in \mathcal B\},
\end{align*}
so that $|\mathcal B| = |\mathcal B_n^+|+|\mathcal B_n^-|$, and $||\mathcal B|| = ||\mathcal B_n^+||+||\mathcal B_n^-|| + |\mathcal B_n^+|$. We define $m_n^+=|\mathcal B_n^+|$, and $m_n^-=|\mathcal B_n^-|$. Now we count pairs $(B,i)$ such that $B\in \mathcal B$, $i\in B$ and $B-i \notin \mathcal B$ --- these are the pairs which contribute to $\mathrm{def}(\mathcal B)$. We obtain
\begin{align*}
\mathrm{def}(\mathcal B) = &|\{B\in \mathcal B,\, i\in[n]: i\neq n,\,n\in B,\,B-i\notin \mathcal B\}|+\\
&|\{B\in \mathcal B,\,i\in [n]: i\neq n,\,n \notin B,\,B-i\notin \mathcal B\}|+\\
&|\{B\in \mathcal B: n \in B,\,B-n\notin \mathcal B\}|\\
&= \mathrm{def}(\mathcal B_n^+) + \mathrm{def}(\mathcal B_n^-)+|\mathcal B_n^+\setminus \mathcal B_n^-|.
\end{align*}
By the induction hypothesis and Lemma \ref{lem_colex_sums},
\begin{align*}
||\mathcal B|| &= ||\mathcal B_n^+|| + ||\mathcal B_n^-|| + |\mathcal B_n^+|\\
&\le ||\mathcal I(m_{n,+})|| + ||\mathcal I(m_{n,-})|| + m_{n,+} + \mathrm{def}(\mathcal B_n^+)+ \mathrm{def}(\mathcal B_n^-)\\
&\le ||\mathcal I(m)|| - \min(m_{n,+},m_{n,-}) + m_{n,+} + \mathrm{def}(\mathcal B) - |\mathcal B_n^+\setminus \mathcal B_n^-|.
\end{align*}
If $m_{n,+}\le m_{n,-}$ then $||\mathcal B||\le ||\mathcal I(m)|| + \mathrm{def}(\mathcal B)$, and so we are done. If not, then since $|\mathcal B_n^+\setminus \mathcal B_n^-| \ge m_{n,+} - m_{n,-}$ we have
\begin{align*}
||\mathcal B|| &\le ||\mathcal I(m)|| - m_{n,-} +m_{n,+} + \mathrm{def}(\mathcal B) - (m_{n,+} - m_{n,-})\\
&= ||\mathcal I(m)|| + \mathrm{def}(\mathcal B),
\end{align*}
and so we are also done.
\end{proof}
\begin{remark}
For positive integers $k$ and $m$, there is a family $\mathcal B$ with $|\mathcal B| =m$ and $\mathrm{def}(\mathcal B) = km$ so that the inequality in Lemma \ref{lem_def} is tight --- we can take $\mathcal B = \{A\cup\{N,\dots,N+k-1\}: A\in \mathcal I(m)\}$, for $N$ a sufficiently large integer. For general $|\mathcal B|$ and $\mathrm{def}(\mathcal B)$, there is not always a family $\mathcal B$ so that the inequality is tight; for example if $|\mathcal B|=2$ and $\mathrm{def}(\mathcal B) = 3$, then in fact $||\mathcal B|| \le 3 = ||\mathcal I(2)|| + 2$.
\end{remark}
Together, Lemmas \ref{lem_no_falls} and \ref{lem_full_sh} show that if $\mathcal B$ has many bad sets then $||\mathcal B||$ is small. Indeed, if $\mathcal B$ has $b$ bad sets then either at least $b/2$ sets of $\mathcal B$ have $\delta B\subseteq \mathcal B$, or at least $b/2$ have $d_{\mathcal B}(B)=B$. By Lemma \ref{lem_full_sh} in the first case, and Lemma \ref{lem_no_falls} in the second,
\[
||\mathcal B||\le ||\mathcal I(|\mathcal B|)||+|\mathcal B|-b/2.
\]
Our aim now is to give a lower bound on the number of bad sets of $\mathcal B$. To do this, we shall focus on how the down-compression $d_{\mathcal B}$ affects the sets of a simply rooted family $\mathcal B$.
\begin{lemma}\label{lem_fall_b}
Let $\mathcal B$ be a simply rooted family, and let $B\in \mathcal B$ with $B-b\notin \mathcal B$ for some $b\in B$. Then $d_{\mathcal B}(B) \in \{B, B-b\}$.
\end{lemma}
\begin{proof}
Since $\mathcal B$ is a simply rooted family, for some $a\in B$ we have $[\{a\},B]\subseteq \mathcal B$. But $B-b\notin \mathcal B$, and so $a=b$. We now consider $D_{b-1}(\mathcal B)$, the family obtained by applying the compressions $d_1,\dots,d_{b-1}$ to $\mathcal B$, starting with $d_1$. We note that the cube $[\{b\},B]$ is fixed when we apply any down-compression $d_i$ with $i\neq b$ to $\mathcal B$; indeed, if $A\in [\{b\},B]$ then $A-i \in [\{b\},B]$, so $d_{(\mathcal B,i)}(A)=A$. Hence we have $[\{b\},B] \subseteq D_{b-1}(\mathcal B)$.

We now consider two cases. If $B-b \in D_{b-1}(\mathcal B)$ then it is $D_{(\mathcal B,b-1)}(B')$ for some $B'\in \mathcal B$. Hence $D_{(\mathcal B,b-1)}(B')\neq B'$, and so by Part \ref{it_down-set} of Lemma \ref{lem_rooted_basics} we have $\mathcal P(B-b)\subseteq D_{b-1}(\mathcal B)$. Since $[\{b\},B]\subseteq D_{b-1}(B)$, we then have $\mathcal P(B)\subseteq D_{b-1}(\mathcal B)$ and so $d_\mathcal B(B)=B$. On the other hand, if $B-b \notin D_{b-1}(\mathcal B)$ then $D_{(\mathcal B,b)}(B)= B-b$, and by Part \ref{it_one_fall} of Lemma \ref{lem_rooted_basics} we have $d_\mathcal B(B) = B-b$.
\end{proof}
In the next lemma, we show that if $\mathcal B'\subseteq \mathcal B$ are simply rooted families, then sets in $\mathcal B'$ which are fixed by $d_{\mathcal B'}$ are also fixed by $d_\mathcal B$.
\begin{lemma}\label{lem_smaller_falls}
Let $\mathcal B'\subseteq \mathcal B$ be simply rooted families, and let $B\in \mathcal B'$ with $d_{\mathcal B'}(B)=B$. Then $d_{\mathcal B}(B)=B$.
\end{lemma}
\begin{proof}
By induction on $k$, it is easy to show that for all $1\le k \le n$ we have $D_k(\mathcal B') \subseteq D_k(\mathcal B)$. Indeed, if $\mathcal F'\subseteq \mathcal F\subseteq \mathcal P(n)$, then for any $i\in [n]$ we have $d_i(\mathcal F')\subseteq d_i(\mathcal F)$. Since $d_{\mathcal B'}(B)=B$, for all $k \in B$ we must have $B-k \in D_{k-1}(\mathcal B')$, and so also $B-k \in D_{k-1}(\mathcal B)$. This is exactly the condition we need to guarantee $d_{\mathcal B}(B)=B$.
\end{proof}
From the previous lemmas, we can read out a result on how the good sets of $\mathcal B$ behave under down-compressions $d_{\mathcal B'}$ for simply rooted families $\mathcal B'\subseteq \mathcal B$.
\begin{lemma}\label{lem_good_fall}
Let $\mathcal B'\subseteq \mathcal B$ be simply rooted families, with $B\in \mathcal B'$ a good set of $\mathcal B$. Then $d_\mathcal B(B) = d_{\mathcal B'}(B)$.
\end{lemma}
\begin{proof}
Since $B$ is a good set of $\mathcal B$, for some $b \in B$ we have $B-b\notin \mathcal B$, and by Part \ref{it_one_fall} of Lemma \ref{lem_rooted_basics} we have $d_\mathcal B(B) = B-a$ for some $a\in B$. So by Lemma \ref{lem_smaller_falls}, $d_{\mathcal B'}(B)\neq B$, and so $d_{\mathcal B'}(B) = B-c$ for some $c\in B$. But by Lemma \ref{lem_fall_b}, $a=c=b$, and the result holds.
\end{proof}

\subsection{Unions of simply rooted families}
Now we are in a position to prove a lemma about simply rooted families $\mathcal B$ which can be decomposed as the union of two other simply rooted families $\mathcal B_1 \cup \mathcal B_2$. Specifically, we show that if $\mathcal B_1$ and $\mathcal B_2$ have small intersection, but the down-sets $d(\mathcal B_1)$ and $d(\mathcal B_2)$ have large intersection, then $\mathcal B$ has many bad sets.

\begin{lemma}\label{lem_split_rooted}
Let $\mathcal B$, $\mathcal B_1$ and $\mathcal B_2$ be simply rooted families, with $\mathcal B_1 \cup \mathcal B_2 = \mathcal B$. Let $b$ be the number of bad sets of $\mathcal B$. Then
\[
|d(\mathcal B_1)\cap d(\mathcal B_2)| \le b + |\mathcal B_1 \cap \mathcal B_2|.
\]
\end{lemma}
\begin{proof}
Let $B$ be a set in $d(\mathcal B_1)\cap d(\mathcal B_2)$. Then $B= d_{\mathcal B_1}(B_1)= d_{\mathcal B_2}(B_2)$, for some $B_1\in \mathcal B_1$ and $B_2\in \mathcal B_2$. If both $B_1$ and $B_2$ are good sets of $\mathcal B$ then, applying Lemma \ref{lem_good_fall},
\[
d_{\mathcal B}(B_1) = d_{\mathcal B_1}(B_1)= d_{\mathcal B_2}(B_2)=d_{\mathcal B}(B_2).
\]
But $d_{\mathcal B}:\mathcal B\to d(\mathcal B)$ is injective, and so $B_1=B_2 \in \mathcal B_1 \cap \mathcal B_2$. On the other hand, if $B_1$ and $B_2$ are not both good sets of $\mathcal B$, $B$ is the $d_{\mathcal B_1}$ image of a bad set of $\mathcal B$ in $\mathcal B_1$, or the $d_{\mathcal B_2}$ image of a bad set of $\mathcal B$ in $\mathcal B_2$. Hence the number of sets in $d(\mathcal B_1)\cap d(\mathcal B_2)$ is at most the number of good sets of $\mathcal B$ in $\mathcal B_1 \cap \mathcal B_2$, plus the number of bad sets of $\mathcal B$ in $\mathcal B_1$, plus the number of bad sets of $\mathcal B$ in $\mathcal B_2$ --- which is precisely $b + |\mathcal B_1 \cap \mathcal B_2|$.\end{proof}
This result has an immediate corollary using Harris's Lemma \cite{Har}, which states that down-sets in the cube are positively correlated. Precisely, if $\mathcal D_1$ and $\mathcal D_2$ are down-sets in $\mathcal P(n)$, then $|\mathcal D_1\cap \mathcal D_2|\ge 2^{-n}|\mathcal D_1||\mathcal D_2|$. Applying this to the down-sets $d(\mathcal B_1)$ and $d(\mathcal B_2)$, and using the fact that $|d(\mathcal B_i)|=|\mathcal B_i|$ for $i=1$ and $2$, we get the following result.
\begin{corollary}\label{cor_lower_b}
Let $\mathcal B$, $\mathcal B_1$ and $\mathcal B_2$ be simply rooted families, with $\mathcal B_1 \cup \mathcal B_2 = \mathcal B$. Let $b$ be the number of bad sets of $\mathcal B$. Then
\[
2^{-n}|\mathcal B_1| |\mathcal B_2| \le b + |\mathcal B_1 \cap \mathcal B_2|.
\]
\qed
\end{corollary}

Next we shall choose simply rooted families $\mathcal B_1$ and $\mathcal B_2$ to which we can apply this result to give a lower bound on the number of bad sets in $\mathcal B$. Recall that for a set $S\subseteq [n]$ and a simply rooted family $\mathcal B$, $\mathcal B_S$ is the family consisting those elements of $B$ which are $\mathcal B$-rooted at some element of $S$. We note that $\mathcal B_S$ is a simply rooted family; if $B$ is $\mathcal B$-rooted at $s\in S$, every set of $[\{s\},B]$ is $\mathcal B$-rooted at $s$ and hence is in $\mathcal B_S$, so $B$ is $\mathcal B_S$-rooted at $s$. We restate Corollary \ref{cor_lower_b} for these simply rooted families.
\begin{lemma}\label{lem_many_bad}
Let $(S,T)$ be a partition of $[n]$ into two disjoint sets, and let $\mathcal B$ be a simply rooted family in $\mathcal P(n)$. Let $b_1$ be the number of sets $B\in \mathcal B\setminus (\mathcal B_S\cap\mathcal B_T)$ with $\delta B\subseteq \mathcal B$, let $b_2 = |\mathcal B_S \cap \mathcal B_T|$, and let $b_3$ be the number of sets $B\in \mathcal B$ with $d_{\mathcal B}(B)=B$. Then
\[
 2^{-n}|\mathcal B_S||\mathcal B_T|\le b_1+2b_2+b_3.
\]
\end{lemma}
\begin{proof}
Note that since $\mathcal B$ is a simply rooted family, $\mathcal B_S$ and $\mathcal B_T$ are simply rooted families and $\mathcal B_S\cup \mathcal B_T = \mathcal B$. Also, there are at most $b_1+b_3$ bad sets of $\mathcal B$ not in $\mathcal B_S \cap \mathcal B_T$, and at most $b_2$ in $\mathcal B_S\cap \mathcal B_T$. Hence the total number of bad sets of $\mathcal B$ is at most $b_1+b_2+b_3$, and so by Corollary \ref{cor_lower_b} we have
\[
2^{-n}|\mathcal B_S||\mathcal B_T| \le b_1+b_2+b_3 + |\mathcal B_S\cap \mathcal B_T| = b_1+2b_2+b_3,
\]
as required.
\end{proof}

We note that in fact every set in $\mathcal B_S \cap \mathcal B_T$ is bad --- indeed, any set $B$ which is $\mathcal B$-rooted at two distinct integers has $\delta B \subseteq \mathcal B$. Hence this result gives us a lower bound on the number of bad sets in $\mathcal B$.

To use Lemma \ref{lem_many_bad} to prove Theorem \ref{thm_stability}, we shall pick $S$ and $T$ to make $|\mathcal B_S||\mathcal B_T|$ large. In general, we cannot do well; if, for example, $m \le 2^{n-1}$ and $\mathcal B$ is $\{B+n: B\in \mathcal I(m)\}$, then for any partition $[n]=S\cup T$ one of $\mathcal B_S$ and $\mathcal B_T$ is empty --- which is as we expect, because this family has no bad sets. However, if $\mathcal B_{\{i\}}$ --- that is, the family of sets of $\mathcal B$ which are $\mathcal B$-rooted at $i$ --- is not too large for any $i$, we can easily choose $S$ and $T$ to make $|\mathcal B_S||\mathcal B_T|$ large.
\begin{lemma}\label{lem_large_product}
Let $\mathcal B$ be a simply rooted family in $\mathcal P(n)$ with $|\mathcal B| = m$. Suppose that no $i \in [n]$ has $|\mathcal B_{\{i\}}| > p m$. Then there exists a partition $[n]=S\cup T$ such that
\[
|\mathcal B_S||\mathcal B_T| \ge m^2(1/4 - p^2/4).
\]
\end{lemma}
\begin{proof}
Take the partition $[n]=S\cup T$ where the smaller of $|\mathcal B_S|$ and $|\mathcal B_T|$ is as large as possible --- without loss of generality $|\mathcal B_S|\le |\mathcal B_T|$. If $|\mathcal B_S|< m(1/2-p/2)$, we can move an element $t$ of $T$ to $S$ such that $\min(\mathcal B_S,\mathcal B_T)$ increases, a contradiction. Since $|\mathcal B_S|+|\mathcal B_T|\ge m$, $|\mathcal B_S||\mathcal B_T| \ge (m/2-p/2)(m/2+p/2)=m^2(1/4 - p^2/4)$.
\end{proof}
We are now ready to prove Theorem \ref{thm_stability}. Let $\mathcal B$ be a simply rooted family in $\mathcal P(n)$ with $|\mathcal B| = m$, such that no $i \in [n]$ has $|\mathcal B_{\{i\}}| > p m$. By Lemma \ref{lem_large_product}, there exists a partition $[n]=S\cup T$ such that $|\mathcal B_S||\mathcal B_T|\ge m^2(1/4 - p^2/4)$. We let $b_1$ be the number of sets $B\in \mathcal B\setminus (\mathcal B_S\cap\mathcal B_T)$ with $\delta B\subseteq \mathcal B$, $b_2 = |\mathcal B_S \cap \mathcal B_T|$, and $b_3$ be the number of sets $B\in \mathcal B$ with $d_{\mathcal B}(B)=B$. Then from Lemma \ref{lem_many_bad} we have
\[
2^{-n}m^2(1/4 - p^2/4)\le 2^{-n}|\mathcal B_S||\mathcal B_T| \le b_1+2b_2+b_3,
\]
and so either $b_1+b_2\ge m^2(1/12 - p^2/12)/2^n$ or $b_3\ge m^2(1/12 - p^2/12)/2^n$. In the first case, since $b_1+b_2$ is the number of sets in $\mathcal B$ with $\delta B\subseteq \mathcal B$, by Lemma \ref{lem_full_sh} we have $||\mathcal B||\le ||\mathcal I(m)|| + m - m^2(1/12 - p^2/12)/2^n$. In the second case, from Lemma \ref{lem_no_falls} we also have $||\mathcal B||\le ||\mathcal I(m)|| + m - m^2(1/12 - p^2/12)/2^n$, as required.\hfill\qed

\subsection{Proof of Theorem \ref{thm_main}}
We now prove Theorem \ref{thm_main} from Theorem \ref{thm_stability}, giving us a tighter restriction than Theorem \ref{thm_colex_bound} on (hypothetical) counterexamples to the union-closed conjecture. Let $\mathcal A$ be such a counterexample, with $\mathcal B=\mathcal P(n)\setminus \mathcal A$ and $|\mathcal B| = m$. Our task is to show that $||\mathcal I(m)||>m(n/2-1+c_1)$, for some universal constant $c_1$. If $||\mathcal B_{\{i\}}||$ is small for all $i$, we shall prove this using Theorem \ref{thm_stability}, since if $\mathcal A$ is a counterexample to the union-closed conjecture we have $mn/2 \le ||\mathcal B_{\{i\}}||$. To complete the proof of Theorem \ref{thm_main}, we shall show that if $|\mathcal B_{\{i\}}|$ is large for some $i$ then $||\mathcal I(m)||$ is large. For this, we use the following simple observation.
\begin{lemma}\label{lem_low_degrees}
Let $\mathcal A\subseteq \mathcal P(n)$ be a counterexample to the union-closed conjecture, let $\mathcal B = \mathcal P(n)\setminus \mathcal A$, and let $p\in [0,1/2]$. If some element of $[n]$ is in $m(1/2+p)$ sets of $\mathcal B$ then
\[
 ||\mathcal I(m)|| > m(n/2 - 1 + p).
\]
\end{lemma}
\begin{proof}
From Equation \eqref{eq_rooted_bound}, we have $||\mathcal B|| \le ||\mathcal I(m)|| + m$. Here, since every element of $[n]$ is in more than $m/2$ sets of $\mathcal B$, we must also have $||\mathcal B|| > (n-1)m/2 + m(1/2+p)$, and the result follows.
\end{proof}
Now we can show that if many sets of $\mathcal B$ are $\mathcal B$-rooted at the same $i\in [n]$ then Theorem \ref{thm_main} holds. We shall use Lemma \ref{lem_low_degrees}, and also Theorem 19 of \cite{BaBoEc}, which we state in a slightly different form.
\begin{theorem}\label{thm_down-set}
Let $\mathcal B\subseteq \mathcal P(n)$ be a simply rooted family with $|\mathcal B|=m$. Suppose the largest down-set contained in $\mathcal B$ is $\mathcal D$. Then $||\mathcal B||\le ||\mathcal I(m)|| + m - |\mathcal D|$.\qed
\end{theorem}
In fact, this theorem is an immediate consequence of Lemma \ref{lem_full_sh}, since every $B\in \mathcal D$ has $\delta B \subseteq \mathcal B$.
\begin{lemma}\label{lem_few_with_root}
Suppose that $\mathcal A\subseteq \mathcal P(n)$ is a counterexample to the union-closed conjecture, let $\mathcal B = \mathcal P(n)\setminus \mathcal A$, and let $p\in [0,1]$. If $|\mathcal B_{\{i\}}| \ge 3p m$ for some $i \in [n]$, then
\[
 ||\mathcal I(m)|| > m(n/2 - 1 + p).
\]
\end{lemma}
\begin{proof}
 We may assume $i=n$. We define
\begin{align*}
\mathcal B^+_n &= \{B\subseteq \mathcal P(n-1):B + n\in \mathcal B\},\\
\mathcal B^-_n &= \{B\subseteq \mathcal P(n-1): B\in \mathcal B\}.
\end{align*}
Also, define $m_{n,+} = |\mathcal B^+_n|$, and $m_{n,-} = |\mathcal B^-_n|$. Since $\mathcal A$ is a counterexample to the union-closed conjecture, $m_{n,+}> m_{n,-}$. If $m_{n,+}>m(1/2+p)$, we are done by Lemma \ref{lem_low_degrees}, so we may assume that $m_{n,+}\le m(1/2+p)$, and hence $m_{n,+}-m_{n,-}\le 2pm$. Then, setting $D_+$ to be the largest down-set contained in $\mathcal B_n^+$, we have $\{B-n: [\{n\},B]\subseteq \mathcal B\}\subseteq \mathcal D_+$, and so $|\mathcal D_+|\ge 3pm \ge m_{n,+}-m_{n,-}+pm$. Applying Theorem \ref{thm_down-set} to $\mathcal B^+_n$ now gives us
\begin{align*}
||\mathcal B||&=||\mathcal B^+_n||+||\mathcal B^-_n||+m_{n,+}\\
&\le ||\mathcal I(m_{n,+})||+m_{n,+}-|\mathcal D_+|+||\mathcal I(m_{n,-})||+ m_{n,-}+m_{n,+}\\
&= ||\mathcal I(m_{n,+})||+||\mathcal I(m_{n,-})|| + m +m_{n,+} - |\mathcal D_+|\\
&\le ||\mathcal I(m_{n,+})||+||\mathcal I(m_{n,-})|| + m + m_{n,-} - pm.
\end{align*}
Now, since $m_{n,+}>m_{n,-}$, by Lemma \ref{lem_colex_sums} we have $||\mathcal I(m_{n,+})||+||\mathcal I(m_{n,-})||+ m_{n,-}\le ||\mathcal I(m)||$, and hence
\[
||\mathcal B|| \le ||\mathcal I(m)||+m-pm.
\]
Since $\mathcal B$ is the complement of a counterexample to the union-closed conjecture, we also have $||\mathcal B|| > mn/2$, and the result follows.
\end{proof}

Putting Theorem \ref{thm_stability} and Lemma \ref{lem_few_with_root} together, we can prove Theorem \ref{thm_main}. Indeed, suppose there is a counterexample $\mathcal A$ to the union-closed conjecture in $\mathcal P(n)$, and let $\mathcal B$ be $\mathcal P(n)\setminus \mathcal A$ with $|\mathcal B| = m$. Suppose that $||\mathcal I(m)|| = m(n/2-1+p)$. Then by Lemma \ref{lem_few_with_root} we have $|\mathcal B_{\{i\}}|\le 3pm$ for every $i\in [n]$. The family $\mathcal B$ is the complement of a union-closed family, and so is simply rooted, so by Theorem \ref{thm_stability} we have
\[
||\mathcal B|| \le ||\mathcal I(m)|| + m - m^2(1/12 - 9p^2/12)/2^n.
\]
However, $||\mathcal B|| > mn/2$, since $\mathcal B$ is the complement of a counterexample to the union-closed conjecture. Hence
\[
||\mathcal I(m)|| = m(n/2-1+p) > m \left(n/2 - 1 + \frac{m(1-9p^2)}{12\cdot 2^n}\right).
\]
Now, $\mathcal P(n)\setminus \mathcal B$ is a counterexample to the union-closed conjecture, so by Corollary \ref{cor_old_bound} we have $m\ge 2^n/3$, and so
\[
pm > m(1/36 - 9p^2/36),
\]
and
\[
36p + 9p^2 > 1.
\]
This is false for all $0\le p \le 1/37$, and so we have that
\[
 ||\mathcal I(m)|| > m(n/2-1+1/37),
\]
proving Theorem \ref{thm_main} with a bound of $c_1\ge 1/37$. \qed

\section{Bounding $||\mathcal I(m)||$}\label{sec_pf_cor}
In this section we bound $||\mathcal I(m)||$, enabling us to prove Corollary \ref{cor_main}. We will use a result of Cz\'edli, Mar\'oti and Schmidt \cite{CzMa}, which states that for a positive integer $r$ we have $||\mathcal I(m)|| > mr/2$ if and only if $m>2^{r+2}/3$. Here, we shall want a more precise bound for general $m$.
\begin{lemma}\label{lem_colex_total}
Let $r$ and $m$ be positive integers with $r\ge 1$ and $2^r/3\le m \le 2^{r+1}/3$, and write $m = 2^r/3 + m'$. Then
\[
 ||\mathcal I(m)|| \le m(r/2-1) + 3m'/2.
\]
\end{lemma}
\begin{proof}
We prove this by induction on $r$ --- we deduce the assertion for $r$ from those for $r-1$ and $r-2$. For $r=1$ or $2$ the result is easy to check. For $r\ge 3$, first suppose that $m \ge 2^{r-1}$. Since $m \le 2^{r+1}/3$, we have
\begin{align*}
||\mathcal I(m)|| \le m(r/2-1/2)= m(r/2-1) + m/2.
\end{align*}
Also, $m'\ge m/3$, so the result follows. Otherwise, write $m = 2^{r-2} + k$, where $2^{r-2}/3 \le k < 2^{r-2}$. If $k \ge 2^{r-1}/3$, we set $k = 2^{r-1}/3 +k'$ and use the induction hypothesis;
\begin{align*}
||\mathcal I(m)|| &= (r/2-1)2^{r-2} + k +||\mathcal I(k)||\\
&\le (r/2-1)2^{r-2} + k+ k(r/2-3/2) + 3k'/2\\
&= (r/2-1)m - k/2 + 3(k-2^{r-1}/3)/2,
\end{align*}
while $m' = k - 2^{r-2}/3$. Hence we need that for all $2^{r-2}/3 \le k < 2^{r-2}$,
\[
3/2( k - 2^{r-2}/3) \le k/2 + 3(k-2^{r-1}/3)/2,
\]
which does indeed hold. Finally, if $k<2^{r-1}/3$ we have $k = 2^{r-2}/3 +m'$, and by the induction hypothesis we have
\begin{align*}
||\mathcal I(m)|| &= (r/2-1)2^{r-2} + k +||\mathcal I(k)||\\
&\le (r/2-1)2^{r-2} + k + k(r/2-2) + 3m'/2\\
&= (r/2-1)m + 3m'/2,
\end{align*}
as required.
\end{proof}
In fact, we have equality in Lemma \ref{lem_colex_total} whenever $m$ is of the form $2^{a}+2^{a-2}+\dots+2^{a-2j} + 2^{a-2j-1}$ for some integers $a$ and $j$ with $a>0$, $j\ge 0$ and $a-2j-1>0$. We can now prove Corollary \ref{cor_main}. If $\mathcal A$ is a counterexample to the union-closed conjecture in $\mathcal P(n)$, and $\mathcal B = \mathcal P(n)\setminus \mathcal A$ with $|\mathcal B|=m$, then write $m = 2^n/3 + m'$. Then from Theorem \ref{thm_main} and Lemma \ref{lem_colex_total} we have
\begin{align*}
m(n/2-1)+3m'/2 &\ge ||\mathcal I(m)||\\
&\ge m(n/2-1+1/37),
\end{align*}
and so $3m'/2 \ge (2^n/3+m')/37$, which rearranges to $m'\ge \frac{2}{327}2^n$, and Corollary \ref{cor_main} follows with a bound of $c_2\ge \frac{2}{327}$.\qed

\section{Improving the constants}\label{sec_refine}
In this section, we give a modification to the arguments in Section \ref{sec_pf_thm} which improves the constants in our main theorems. To do this, we give stronger versions of Lemmas \ref{lem_split_rooted} and \ref{lem_many_bad}.
For a triple of simply rooted families $\mathcal B$, $\mathcal B_1$, $\mathcal B_2$ with $\mathcal B=\mathcal B_1\cup \mathcal B_2$,
\[
Z(\mathcal B, \mathcal B_1, \mathcal B_2) = \{B\in \mathcal B_1\cap \mathcal B_2: d_{\mathcal B}(B),\, d_{\mathcal B_1}(B) \textrm{ and } d_{\mathcal B_2}(B) \textrm{ are all distinct}\}.
\]
The definition of $Z(\mathcal B, \mathcal B_1, \mathcal B_2)$ is motivated by the proof of Lemma \ref{lem_split_rooted}. The sets in $Z(\mathcal B,\mathcal B_1,\mathcal B_2)$ are those sets $B$ for which $d_{\mathcal B_1}(B)$ and $d_{\mathcal B_2}(B)$ may be distinct sets of $d(\mathcal B_1)\cap d(\mathcal B_2)$, so if we can bound $|Z(\mathcal B,\mathcal B_1,\mathcal B_2)|$ we can improve our bound on $|d(\mathcal B_1)\cap d(\mathcal B_2)|$.
\begin{lemma}\label{lem_split_rooted_2}
Let $\mathcal B$, $\mathcal B_1$ and $\mathcal B_2$ be simply rooted families, with $\mathcal B_1 \cup \mathcal B_2 = \mathcal B$. Let $b$ be the number of bad sets of $\mathcal B$. If every set $B\in \mathcal B_1 \cap \mathcal B_2$ has $\delta B \subseteq \mathcal B$, then 
\[
|d(\mathcal B_1)\cap d(\mathcal B_2)| \le b + |Z(\mathcal B, \mathcal B_1, \mathcal B_2)|.
\]
\end{lemma}
\begin{proof}
The proof is similar to that of Lemma \ref{lem_split_rooted}. Letting $\mathcal D_i=d(\mathcal B_i)$, consider an element $S$ of $\mathcal D_1\cap \mathcal D_2$. Then $S=d_{\mathcal B_1}(B_1)=d_{\mathcal B_2}(B_2)$ for some $B_1\in \mathcal B_1$ and $B_2 \in \mathcal B_2$. We now define a function $f:\mathcal D_1\cap \mathcal D_2\to \mathcal B$. If $B_1=B_2$, then we set $f(S)=B_1$ --- note that since $B_1\in \mathcal B_1\cap \mathcal B_2$, $\delta B \subseteq \mathcal B$ and so $B_1$ is a bad set of $\mathcal B$. Otherwise, since $d_{\mathcal B}$ is injective, for $i=1$ or $2$ we have $d_{\mathcal B_i}(B_i)\neq d_{\mathcal B}(B_i)$. In this case, we define $f(S)=B_i$ --- note that since $d_{\mathcal B_i}(B_i)\neq d_{\mathcal B}(B_i)$, by Lemma \ref{lem_good_fall} the set $B_i$ is a bad set of $\mathcal B$. So $f(S)$ is a bad set of $\mathcal B$ for all $S\in \mathcal D_1\cap \mathcal D_2$. Also, for $S\neq T \in \mathcal D_1\cap \mathcal D_2$, if $f(S) = f(T)=B$ then
\[
S=d_{\mathcal B_i}(B)\neq d_{\mathcal B}(B) \neq d_{\mathcal B_j}(B)=T,
\]
where $\{i,j\}=\{1,2\}$. In particular $B\in Z(\mathcal B,\mathcal B_1, \mathcal B_2)$, and there is no $U\in \mathcal D_1\cap \mathcal D_2$ with $S\neq U\neq T$ and $f(U)=B$. Hence the size of the image of $f$ is at least $|\mathcal D_1\cap \mathcal D_2|-|Z(\mathcal B,\mathcal B_1,\mathcal B_2)|$, and since every set in the image is a bad set of $\mathcal B$ the result follows.
\end{proof}

We can now prove a stronger form of Lemma \ref{lem_many_bad}, using Lemma \ref{lem_split_rooted_2} and making sure we do not overcount the bad sets of $\mathcal B$. For a family of sets $\mathcal B$ we define
\[
Y(\mathcal B) = \{B\in \mathcal B: \delta B\subseteq \mathcal B \textrm{ and } d_{\mathcal B}(B)=B\}.
\]
The sets in $Y(\mathcal B)$ are those that satisfy both criteria for a set to be bad; we have often overcounted the number of bad sets of $\mathcal B$ by $|Y(\mathcal B)|$.

\begin{lemma}\label{lem_many_bad_2}
Let $(S,T)$ be a partition of $[n]$ into two disjoint sets, and $\mathcal B$ be a simply rooted family in $\mathcal P(n)$. Let $b_1$ be the number of sets $B\in \mathcal B\setminus (\mathcal B_S\cap\mathcal B_T)$ with $\delta B\subseteq \mathcal B$, $b_2 = |\mathcal B_S \cap \mathcal B_T|$, and $b_3$ be the number of sets $B\in \mathcal B$ with $d_{\mathcal B}(B)=B$. Then
\[
 2^{-n}|\mathcal B_S||\mathcal B_T|\le b_1 + b_2 + b_3 + |Z(\mathcal B, \mathcal B_S, \mathcal B_T)| - |Y(\mathcal B)|.
\]
\end{lemma}
\begin{proof}
The proof is identical to that of Lemma \ref{lem_many_bad} --- letting $b$ be the number of bad sets of $\mathcal B$, by Harris's Lemma and Lemma \ref{lem_split_rooted_2} we have
\[
2^{-n}|\mathcal B_S||\mathcal B_T|\le b + |Z(\mathcal B, \mathcal B_S, \mathcal B_T)|,
\]
and $b = b_1+b_2+b_3-|Y(\mathcal B)|$.
\end{proof}

We shall show that in fact $|Y(\mathcal B)| \ge |Z(\mathcal B, \mathcal B_S, \mathcal B_T)|$, improving our bound on the number of bad sets of $\mathcal B$. For this, we shall use the key lemma of Reimer \cite{Rei} on up-compressions of union-closed families. For a family $\mathcal A\subseteq \mathcal P(n)$, a set $A\in \mathcal A$ and an element $i\in [n]$, we define
\[
u_{(i,\mathcal A)}(A)=
\begin{cases}
A+i: i\notin A,\, A+i\notin \mathcal A\\
A: \mathrm{ otherwise.}
\end{cases}
\]
Then $u(\mathcal A)$, $u_{\mathcal A}(A)$, $u_i(\mathcal A)$, $U_i(\mathcal A)$ and $U_{(\mathcal A,i)}(A)$ are defined analagously to in the case of down-compressions. In particular, $u(\mathcal A) = u_1\dots u_n (\mathcal A)$, and $u_{\mathcal A}(A)$ is the image of the set $A$ in $u(\mathcal A)$ under the sequence of up-compressions $u_1\dots u_n$.

\begin{lemma}\label{lem_Rei}
If $\mathcal A$ is a union-closed family, and $A_1\neq A_2$ are sets in $\mathcal A$, the cubes $[A_1,u_{\mathcal A}(A_1)]$ and $[A_2,u_{\mathcal A}(A_2)]$ are disjoint.\qed
\end{lemma}

We make a simple observation about the relationship between sets of a simply rooted family which lose an element under the down-compression $d_{\mathcal B}$, and the sets of the union-closed family $\mathcal P(n)\setminus \mathcal B$.
\begin{lemma}\label{lem_uc_image}
Let $\mathcal B\subseteq \mathcal P(n)$ be a simply rooted family, let $\mathcal A = \mathcal P(n)\setminus \mathcal B$, and let $B\in \mathcal B$. If $d_{\mathcal B}(B)\neq B$ then for some $1\le k \le n$ and some $A\in \mathcal A$ we have $U_{(\mathcal A,k)}(A)=B$.
\end{lemma}
\begin{proof}
Let $k$ be minimal with $D_{(\mathcal B,k)}(B)\neq B$. Then $D_{(\mathcal B,k)}(B) = B-k$, and $B-k \notin D_{k-1}(B)$. Hence $B-k \in \mathcal P(n)\setminus D_{k-1}(B) = U_{k-1}(\mathcal A)$, and so $B-k = U_{(\mathcal A,k-1)}(A)$ for some $A\in \mathcal A$, and $B = U_{(\mathcal A,k)}(A)$.
\end{proof}
For a simply rooted family $\mathcal B$, and a set $B\in \mathcal B$, let $R_{\mathcal B}(B)=\{r\in [n]:[\{r\},B]\subseteq \mathcal B\}$ be the set of roots of $B$ in $\mathcal B$. We now prove that if $B$ is in some cube $[A,U_\mathcal A(A)]$, we must have $A= B\setminus R_{\mathcal B}(B)$.
\begin{lemma}\label{lem_cube_set}
Let $\mathcal B\subseteq \mathcal P(n)$ be a simply rooted family, and let $B\in \mathcal B$. Let $\mathcal A= \mathcal P(n)\setminus \mathcal B$. If $B$ is in the cube $[A,U_\mathcal A(A)]$ for some $A\in \mathcal A$, then $A= B\setminus R_\mathcal B(B)$.
\end{lemma}
\begin{proof}
Let $R= R_\mathcal B(B)$. First we observe that $B\setminus R \in \mathcal A$. Indeed, suppose $B\setminus R\in \mathcal B$; then it is $\mathcal B$-rooted at some $b\in B\setminus R$. But then we have $\{B'\subseteq B: B'\cap R \neq \emptyset\}\subseteq \mathcal B$, and $\{B'\subseteq B\setminus R: b\in B'\}\subseteq \mathcal B$. Hence $B$ is $\mathcal B$-rooted at $b$, and so $b\in R$, a contradiction as $b\in B\setminus R$.

Now, since $B\in [A,U_\mathcal A(A)]$, we have $A\subseteq B$. However, $\{B'\subseteq B: B'\cap R \neq \emptyset\}$ is contained in the family $\mathcal B$, and hence we have $A\subseteq B\setminus R$. In particular, $B\setminus R \in [A,B]\subseteq [A,U_\mathcal A(A)]$. Hence the cubes $[A,U_\mathcal A(A)]$ and $[B\setminus R,U_\mathcal A(B\setminus R)]$ intersect, and so from Theorem \ref{lem_Rei} we have $A= B\setminus R$.
\end{proof}
Using the last two lemmas, it is immediate that if a set $B\in \mathcal B$ loses an element $r$ under the down-compression $d_\mathcal B$, then $B$ is $\mathcal B$-rooted at $r$.
\begin{lemma}\label{lem_root_fall}
Let $\mathcal B\subseteq \mathcal P(n)$ be a simply rooted family, and let $B\in \mathcal B$. Then $d_{\mathcal B}(B)\in \{B\}\cup\{B-r:r\in R_{\mathcal B}(B)\}$.
\end{lemma}
\begin{proof}
Let $\mathcal A = \mathcal P(n)\setminus \mathcal B$. Suppose $d_{\mathcal B}(B)\neq B$ --- then by Lemma \ref{lem_rooted_basics}, $d_{\mathcal B}(B)=B-b$ for some $b\in B$. Also, by Lemma \ref{lem_uc_image} we have that for some $k$ and some $A\in \mathcal A$ we have $U_{(\mathcal A, k)}(A)=B$. In particular, $B\in [A,u_{\mathcal A}(A)]$, and so $A=B\setminus R_{\mathcal B}(B)$. Since $d_{\mathcal B}(B) \in [A,B]$, we then have $b \in B\setminus A = R_{\mathcal B}(B)$, as required.
\end{proof}
In the special case where $B-b\notin \mathcal B$ for some $b\in B$, we must have $R_{\mathcal B}(B)=\{b\}$, and so Lemma \ref{lem_fall_b} is a special case of Lemma \ref{lem_root_fall}. We read out the following corollary on the number of roots of sets in $Z(\mathcal B, \mathcal B_S, \mathcal B_T)$.
\begin{corollary}\label{cor_Z_roots}
Let $\mathcal B$ be a simply rooted family, $S\cup T$ a partition of $[n]$, and $B\in Z(\mathcal B, \mathcal B_S, \mathcal B_T)$. Then $|R_{\mathcal B}(B)|\ge 2$. If $d_{\mathcal B}(B)\neq B$, then $|R_{\mathcal B}(B)|\ge 3$.
\end{corollary}
\begin{proof}
By Lemma \ref{lem_root_fall} the sets $d_{\mathcal B_S}(B)$, $d_{\mathcal B_T}(B)$ and $d_{\mathcal B}(B)$ are all elements of $\{B\}\cup\{B-r:r\in R_{\mathcal B}(B)\}$. But $B\in Z(\mathcal B, \mathcal B_S, \mathcal B_T)$, so these sets are all distinct, and in particular, $|R_{\mathcal B}(B)|\ge 2$. If $d_{\mathcal B}(B)\neq B$, then by Lemma \ref{lem_smaller_falls} we also have $d_{\mathcal B_1}(B)\neq B\neq d_{\mathcal B_2}(B)$, so the sets $d_{\mathcal B_1}(B)$, $d_{\mathcal B_2}(B)$ and $d_{\mathcal B}(B)$ are distinct elements of $\{B-r:r\in R_{\mathcal B}(B)\}$ and $|R_{\mathcal B}(B)|\ge 3$.\end{proof}

Now we shall prove that $|Y(\mathcal B)| \ge |Z(\mathcal B, \mathcal B_S, \mathcal B_T)|$. For a finite set $B$ we define the \emph{$2$nd shadow of $B$} to be $\delta_2 B = \{B'\subseteq B: |B'|=|B|-2\}$.

\begin{lemma}\label{lem_Y_ge_Z}
Let $\mathcal B$ be a simply rooted family, and $S\cup T$ a partition of $[n]$. Then $|Y(\mathcal B)| \ge |Z(\mathcal B, \mathcal B_S, \mathcal B_T)|$.
\end{lemma}
\begin{proof}
We write $Z= Z(\mathcal B, \mathcal B_S, \mathcal B_T)$, and $Y=Y(\mathcal B)$. Let $\mathcal A$ be $\mathcal P(n)\setminus \mathcal B$; $\mathcal A$ is a union-closed family, since $\mathcal B$ is simply rooted. If a set $B\in Z$ is not in a cube $[A,u_{\mathcal A}(A)]$ for some $A\in \mathcal A$, then $B$ is also in $Y$. Indeed, $\delta B\subseteq \mathcal B$ because all sets in $Z$ are $\mathcal B$-rooted at two distinct elements of $[n]$. $d_{\mathcal B}(B)=B$ follows from Lemma \ref{lem_uc_image}; otherwise we must have $B=U_{(\mathcal A,k)}(A)$ for some $A\in \mathcal A$ and $1\le k \le n$, so $B\in [A,u_{\mathcal A}(A)]$, a contradiction. Hence it is enough to show that for every cube $C=[A,u_{\mathcal A}(A)]$,
\[
|C\cap Y| \ge |C\cap Z|,
\]
since by Theorem \ref{lem_Rei} these cubes are disjoint for different $A$. We shall now show this for the cube $C$. If $C\cap Z\subseteq C\cap Y$, we are done. Otherwise, let $B\in (C\cap Z) \setminus Y$. Since $B\in Z$, $B$ has at least two roots in $\mathcal B$ by Corollary \ref{cor_Z_roots}, so $\delta B\subseteq \mathcal B$. Then since $B\notin Y$, $d_{\mathcal B}(B)\neq B$, and $B$ has at least $3$ roots in $\mathcal B$ by Corollary \ref{cor_Z_roots}. Hence $\delta_2(B)\subseteq \mathcal B$, and so $|u_{\mathcal A}(A)\setminus A| \ge |B\setminus A| \ge 3$. We define $r=|u_{\mathcal A}(A)\setminus A|$.

We now count the sets of $C\setminus Y$. Note that by Lemma \ref{lem_Rei}, $C$ contains no set of $\mathcal A$ other than $A$. In $C$, there are $r+1$ sets which are $U_{(\mathcal A,k)}(A)$ for some $0\le k \le n$, one of size $i$ for each $i$ with $|A|\le i \le |u_{\mathcal A}(A)|$. All other sets $B\in C$ are in $\mathcal B$, and have $d_{\mathcal B}(B)=B$. Also, every set $B$ in $C$ of size at least $|A|+2$ has $\delta B\subseteq \mathcal B$. Indeed, if $i\in B\cap A$, $B - i \in C\setminus\{A\}\subseteq \mathcal B$. If $i\in B\setminus A$, then $(B-i)\cup A = B$, and $\mathcal A$ is union-closed, so $B-i\in \mathcal B$.

Hence $|C\cap Y| = 2^r - 2r$ --- the elements of $C\setminus Y$ are precisely the $r+1$ sets of $C$ of size $|A|$ or $|A|+1$, together with one set of size $i$ for each $i$ with $|A|+2\le i \le |A|+r$. To bound $|C\cap Z|$, we note that $A$ is not in $Z$, and nor is $A+i$ for any $i\in (u_{\mathcal A}(A)\setminus A)$, since $A+i$ has only one $\mathcal B$-root. Also, if $A+i+j$ is in $Z$, for $i\neq j$ both in $u_{\mathcal A}(A)\setminus A$, then since $A+i+j$ is not $\mathcal B$-rooted at any element in $A$ by Corollary \ref{cor_Z_roots} we must have $\{d_{\mathcal B}(A+i+j),d_{\mathcal B_S}(A+i+j), d_{\mathcal B_T}(A+i+j)\} = \{A+i+j,A+i,A+j\}$.

However, we must have $d_{\mathcal B}(A+i+j) = A+i+j$; otherwise by Lemma \ref{lem_smaller_falls} $d_{\mathcal B_S}(A+i+j)\neq A+i+j \neq d_{\mathcal B_T}(A+i+j)$, a contradiction. So $\{d_{\mathcal B_S}(A+i+j), d_{\mathcal B_T}(A+i+j)\} = \{A+i,A+j\}$. Without loss of generality, $d_{\mathcal B_S}(A+i+j)=A+i$. Then by Lemma \ref{lem_root_fall} we have $j\in R_{\mathcal B_S}(A+i+j)$, and so $j \in S$. Similarly, $i\in T$.

Now, suppose $A+i+k\in Z$ for some $k\neq j$. Then, as before, $\{d_{\mathcal B_S}(A+i+k), d_{\mathcal B_T}(A+i+k)\} = \{A+i,A+k\}$, and since $i\in T$ we must have $k\in S$ and $d_{\mathcal B_S}(A+i+k)=A+i$, contradicting the injectivity of $d_{\mathcal B_S}$. Hence each element of $u_{\mathcal A}(A)\setminus A$ appears in at most one of size $|A|+2$ in $C\cap Z$. So $C\cap Z$ does not contain $A$, nor any of the $r$ sets of size $|A|+1$ in $C$, and contains at most $\lfloor r/2\rfloor$ of the $\binom{r}{2}$ sets of size $|A|+2$ in $C$. So the total number of sets in $C\cap Z$ is at most $2^r-1-r-\binom{r}{2} +\lfloor r/2\rfloor$. It is easy to see that for $r\ge 3$ this is at most $2^r - 2r$, with equality when $r=3$. Hence $|Y|\ge |Z|$, as required.\end{proof}
\noindent
Combining Lemmas \ref{lem_many_bad_2} and \ref{lem_Y_ge_Z}, we get the following lemma.
\begin{lemma}\label{lem_refinement}
Let $(S,T)$ be a partition of $[n]$ into two disjoint sets. Also, let $\mathcal B$ be a simply rooted family in $\mathcal P(n)$. Let $b_1$ be the number of sets $B\in \mathcal B\setminus (\mathcal B_S\cap\mathcal B_T)$ with $\delta B\subseteq \mathcal B$, let $b_2 = |\mathcal B_S \cap \mathcal B_T|$, and let $b_3$ be the number of sets $B\in \mathcal B$ with $d_{\mathcal B}(B)=B$. Then
\[
 2^{-n}|\mathcal B_S||\mathcal B_T|\le b_1 + b_2 + b_3.
\]
\hfill\qed
\end{lemma}
This result is a stronger version Lemma \ref{lem_many_bad}, and using it instead of that lemma improves the constant in Theorem \ref{thm_stability}, giving the following result.
\begin{theorem}\label{thm_stability_2}
Let $\mathcal B$ be a simply rooted family in $\mathcal P(n)$ with $|\mathcal B| = m$, and let $p\in[0,1]$. Suppose that no $i \in [n]$ has $|\mathcal B_{\{i\}}| \ge pm$. Then
\[
||\mathcal B|| \le ||\mathcal I(m)|| + m - m^2(1/8 - p^2/8)/2^n.
\]
\end{theorem}
\noindent
\begin{proof}
Indeed, by Lemma \ref{lem_large_product} we can choose a partition $[n]=S\cup T$ so that $|\mathcal B_S||\mathcal B_T|\ge m^2(1/2-p^2/4)$. Then we have
\[
2^{-n}m^2(1/2-p^2/4)\le 2^{-n}|\mathcal B_S||\mathcal B_T|\le b_1+b_2+b_3,
\]
and so either $b_1+b_2$ or $b_3$ is at least $m^2(1/8-p^2/8)/2^n$. Applying Lemma \ref{lem_full_sh} in the first case or Lemma \ref{lem_no_falls} in the second, we get Theorem \ref{thm_stability_2}.\end{proof}

This in turn improves the constants in Theorem \ref{thm_main} and Corollary \ref{cor_main}. We also note another minor change to the proof of Theorem \ref{thm_main} --- at the end of the proof, we use the fact that a counterexample to the union-closed conjecture in $\mathcal P(n)$ has fewer than $\frac{2}{3}2^n$ elements. Since we now have a better bound, we can use this instead to improve the argument slightly. Applying these improvements together improves our bound in Theorem \ref{thm_main} to $c_1\ge 0.04218\ldots > 1/24$ and in Corollary \ref{cor_main} to $c_2 \ge 0.009646\ldots > 1/104$ --- that is, the union-closed conjecture holds for families in $\mathcal P(n)$ with at least $(\frac{2}{3}-\frac{1}{104})2^n$ elements.

\section{Further Work}
Theorem \ref{thm_stability_2} is a stability result for the total sizes of simply rooted families, which in turn provides a stability result for the union-closed size problem in the case of large union-closed families; if $\mathcal A\subseteq \mathcal P(n)$ is union-closed, with $|\mathcal A|\ge 2^{n-1}$ and $||\mathcal A||$ is close to the minimum possible, then $\mathcal P(n)\setminus \mathcal A$ has an element of high degree. However, we have no stability result for the union-closed size problem in general. It was proved in \cite{BaBoEc} that there is a unique uinon-closed family $\mathcal F_m$ with $|\mathcal F_m|=m$ and $||\mathcal F_m|| = f(m)$, but if $\mathcal A$ is a union-closed family of $m$ sets with $||\mathcal A||$ close to $||\mathcal F_m||$ in a large powerset, we have no result (or even conjecture) which states that $\mathcal A$ is in some sense similar to $\mathcal F_m$.

Another direction would be to improve our stability results for the sizes of simply rooted families. For example, it was conjectured in \cite{BaBoEc} that if $\mathcal B\subseteq \mathcal P(n)$ is a simply rooted family then
\begin{equation}\label{eq:con1}
||\mathcal B|| \le ||\mathcal I(m)|| + \max_{i\in[n]} \mathrm{deg}_\mathcal B(i).
\end{equation}
This remains open, but we conjecture a stronger result still; that we can replace the maximum of the degrees $d_\mathcal B(i)$ with the largest number of elements of $\mathcal B$ rooted at a single element of $[n]$:
\begin{conjecture}\label{con:max_rooted}
Let $\mathcal B$ be a simply rooted family in $\mathcal P(n)$ then
\[
||\mathcal B|| \le ||\mathcal I(m)|| + \max_{i\in[n]} |\mathcal B_{\{i\}}|.
\]
\end{conjecture}
Even if these conjectures do not hold, it seems likely that some version of Theorem \ref{thm_stability} which does not depend on $n$ is true. To be precise, we conjecture that there are some positive constants $\epsilon$ and $\delta$ such that if $\mathcal B\subseteq \mathcal P(n)$ is a simply rooted family of $m$ sets, and $|\mathcal B_{\{i\}}|\le \epsilon m$ for all $i\in [n]$, then
\[
||\mathcal B|| \le ||\mathcal I(m)|| + m(1-\delta).
\]

\section{Acknowledgements}
The author would like to thank B\'ela Bollob\'as for his helpful comments on earlier versions of this paper.

\end{document}